\documentclass{amsart}

\usepackage{latexsym,amsfonts} 
\usepackage{amsmath,amssymb,graphics,setspace}
\usepackage{amsthm}
\usepackage{MnSymbol}
\usepackage{mathrsfs} 
\usepackage{fontenc}%

\usepackage{marvosym}
\usepackage{paralist}
\usepackage{color}

\usepackage{pstricks,pst-text,pst-grad,pst-node,pst-3dplot,pstricks-add,pst-poly} 
\usepackage{pst-solides3d} 

\usepackage{graphicx}
\usepackage{fontenc}%

\usepackage{hyperref}  
\hypersetup{linktocpage=true,colorlinks=true,linkcolor=blue,citecolor=blue,pdfstartview={XYZ 1000 1000 1}}

\usepackage[verbose]{wrapfig}

\usepackage{subfigure}
\renewcommand{\thesubfigure}{\thefigure.\arabic{subfigure}}
\makeatletter
\renewcommand{\p@subfigure}{}
\renewcommand{\@thesubfigure}{\thesubfigure:\hskip\subfiglabelskip}
\makeatother

\newtheorem{theorem}{Theorem}
\newtheorem{lemma}{Lemma}
\newtheorem{corollary}{Corollary}
\newtheorem{proposition}{Proposition}

\newtheorem{remark}{Remark}
\theoremstyle{definition}
\newtheorem{definition}{Definition}

\newtheorem{example}{Example}
\newtheorem{conjecture}{Conjecture}

\newcommand{\norm}[1]{\left\|#1\right\|}
\newcommand{\str}{\mbox{str}} 
\newcommand{\strShape}{\mbox{strShape}} 
\newcommand{\sheet}{\mbox{sheet}} 
\newcommand{\cyl}{\mbox{cyl}} 
\newcommand{\wsh}{\mbox{wsh}} 
\newcommand{\sh}{\mbox{sheet}} 
\newcommand{\strBUT}{\mbox{strBUT}} 
\newcommand{\near}{\delta} 
\newcommand{\dnear}{\delta_{\Phi}} 
\newcommand{\dfar}{{\not\delta}_{\Phi}} 
\newcommand{\notfar}{\mathop{\not{\delta}}\limits^{\doublewedge}} 
\newcommand{\dcap}{\mathop{\cap}\limits_{\Phi}} 
\newcommand{\sn}{\mathop{\delta}\limits^{\doublewedge}} 
\newcommand{\snd}{\mathop{\delta_{_{\Phi}}}\limits^{\doublewedge}} 
\newcommand{\Int}{\mbox{int}}
\newcommand{\sdfar}{\stackrel{\not{\text{\normalsize$\delta$}}}{\text{\tiny$\doublevee$}}_{_{\Phi}}} 

\newcommand{\sfar}{\stackrel{\not{\text{\normalsize$\delta$}}}{\text{\tiny$\doublevee$}}} 

\begin{document}

\title{String-Based Borsuk-Ulam Theorem and Wired Friend Theorem}

\author[J.F. Peters]{J.F. Peters$^{\alpha}$}
\email{James.Peters3@umanitoba.ca}
\address{\llap{$^{\alpha}$\,}Computational Intelligence Laboratory,
University of Manitoba, WPG, MB, R3T 5V6, Canada and
Department of Mathematics, Faculty of Arts and Sciences, Ad\.{i}yaman University, 02040 Ad\.{i}yaman, Turkey}
\author[A. Tozzi]{A. Tozzi$^{\beta}$}
\address{\llap{$^{\beta}$\,} Department of Physics, Center for Nonlinear Science, University of North Texas, Denton, Texas 76203 and Computational Intelligence Laboratory, University of Manitoba, WPG, MB, R3T 5V6, Canada}
\thanks{The research has been supported by the Scientific and Technological Research
Council of Turkey (T\"{U}B\.{I}TAK) Scientific Human Resources Development
(BIDEB) under grant no: 2221-1059B211402463, Natural Sciences \&
Engineering Research Council of Canada (NSERC) discovery grant 185986 and Instituto Nazionale di Alta Matematica (INdAM) Francesco Severi, Gruppo Nazionale per le Strutture Algebriche, Geometriche e Loro Applicazioni grant 9 920160 000362, n.prot U 2016/000036.}


\date{}

\dedicatory{Dedicated to Som Naimpally, Mathematician and D.I. Olive, F.R.S.}

\begin{abstract}
This paper introduces a string-based extension of the Borsuk-Ulam Theorem (denoted by strBUT).  A string is a region with zero width and either bounded or unbounded length on the surface of an $n$-sphere or a region of a normed linear space.  In this work, an $n$-sphere surface is covered by a collection of strings. For a strongly proximal continuous function on an $n$-sphere into $n$-dimensional Euclidean space, there exists a pair of antipodal $n$-sphere strings with matching descriptions that map into Euclidean space $\mathbb{R}^n$.  Each region $M$ of a string-covered $n$-sphere is a worldsheet.  For a strongly proximal continuous mapping from a worldsheet-covered $n$-sphere to $\mathbb{R}^n$, strongly near antipodal worldsheets map into the same region in $\mathbb{R}^n$.  This leads to a wired friend theorem in descriptive string theory.  An application of strBUT is given in terms of the evaluation of Electroencephalography (EEG) patterns.
\end{abstract}


\maketitle

\section{Introduction}
In this paper, we consider a geometric structure that has the characteristics of a cosmological string $A$ (denoted $\str A$), which is the path followed by a particle $A$ moving through space. 

\setlength{\intextsep}{0pt}
\begin{wrapfigure}[11]{R}{0.33\textwidth}
\begin{minipage}{3.3 cm}
\centering
\includegraphics[width=35mm]{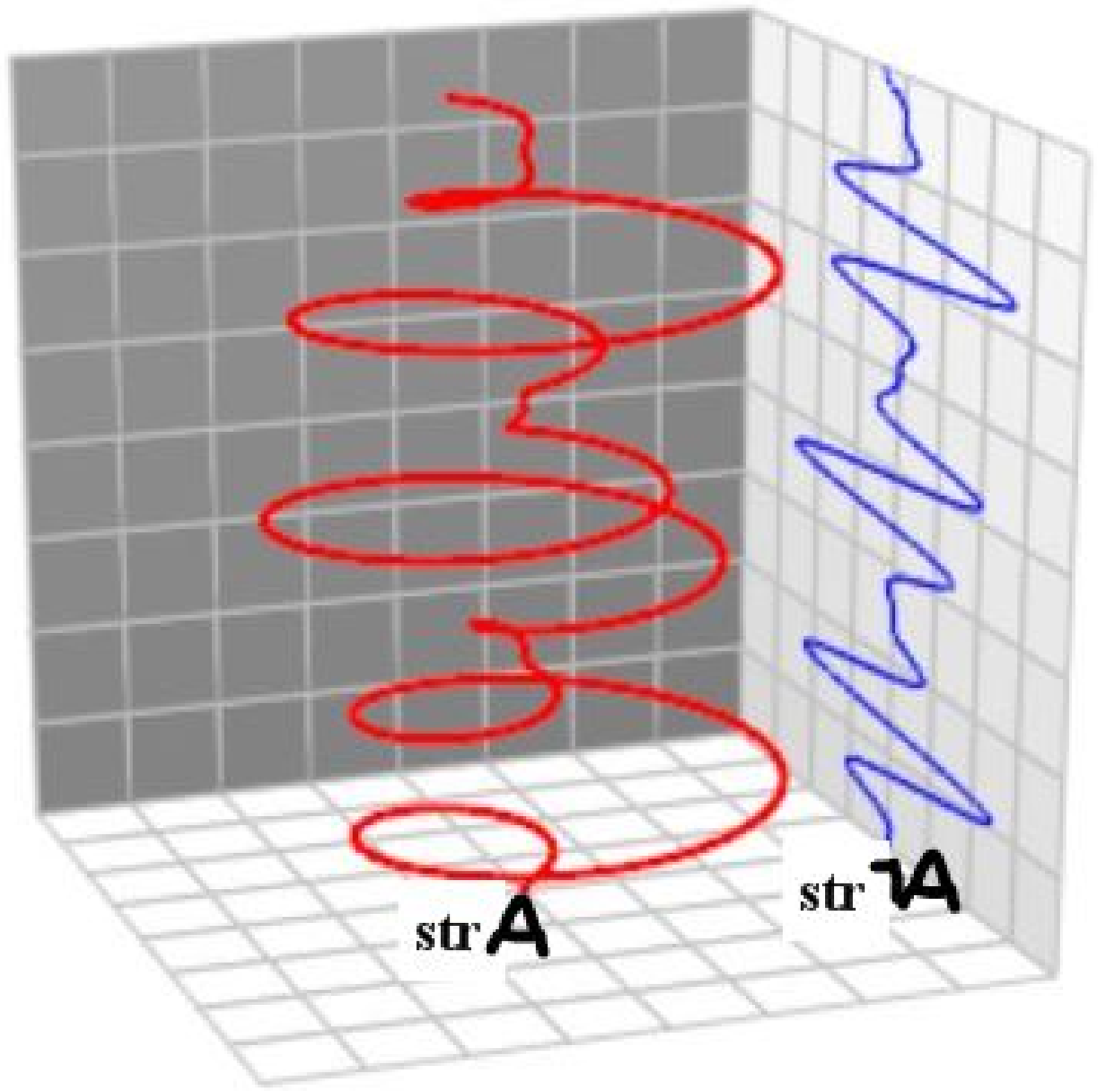}
\caption[]{\footnotesize \bf Strings}
\label{fig:worldlinesABE}
\end{minipage}
\end{wrapfigure}
$\mbox{}$\\

 A \emph{string} is a region of space with zero width (in an abstract geometric space or geometry in a science of spaces~\cite[\S 8.0]{ScribaSchreiber2005geometry}, esp.~\cite{Brouwer1954CJMpointsAndSpaces}) or non-zero width (in a non-abstract, physical geometry space such as the space described by V.F. Lenzen~\cite{Lenzen1939AMMphysicalGeometry}) and either bounded or unbounded length.  Another name for such a string is \emph{worldline}~\cite{OliveLandsberg1989stringTheory,Olive1988stringsAndSuperstrings,Olive1987algebrasAndStrings}.  Here, a string is a region on the surface of an $n$-sphere or in an $n$-dimensional normed linear space.  Every string is a spatial region, but not every spatial region is a string.  Strings that have some or no points in common are \emph{antipodal}.  

Examples of strings with finite length are shown in Fig.~\ref{fig:worldlinesABE}.  Regions $\str A,\righthalfcap \str A$ are examples of antipodal strings.
Various continuous mappings from an $n$-dimensional hypersphere $S^n$ to a feature space that is an $n$-dimensional Euclidean space $\mathbb{R}^n$ lead to a string-based incarnation of the Borsuk-Ulam Theorem, which is a remarkable finding about continuous mappings from antipodal points on an $n$-sphere into $\mathbb{R}^n$, found by K. Borsuk~\cite{Borsuk1933FMantipodes}.


The Borsuk-Ulam Theorem is given in the following form by M.C. Crabb and J. Jaworowski~\cite{Crabb2013BUT}.

\begin{theorem}\label{thm:Borsuk-Ulam1933one}{\bf Borsuk-Ulam Theorem}~{\rm\cite[Satz II]{Borsuk1933FMantipodes}}.\\
Let $f:S^n\longrightarrow \mathbb{R}^n$ be a continuous map.  There exists a point $x\in S^n\subseteq\mathbb{R}^{n+1}$ such that $f(x) = f(-x)$.  
\end{theorem}
\begin{proof}
A proof of this form of the Borsuk-Ulam Theorem follows from a result given in 1930 by L.A. Lusternik and L. Schnirelmann~\cite{Lusternik1930BUT}.  For the backbone of the proof, see Remark 3.4 in~\cite{Crabb2013BUT}.
\end{proof}

\section{Preliminaries}
This section briefly introduces Petty antipodal sets that are strings and the usual continuous function required by BUT is replaced by a proximally continuous function.  In addition, we briefly consider strong descriptive proximities as well as two point-based variations of the Borsuk-Ulam Theorem (BUT) for topological spaces equipped with a descriptive proximity or a strong descriptive proximity.
$\mbox{}$\\
\vspace{3mm}

\begin{figure}[!ht]
\centering
\includegraphics[width=55mm]{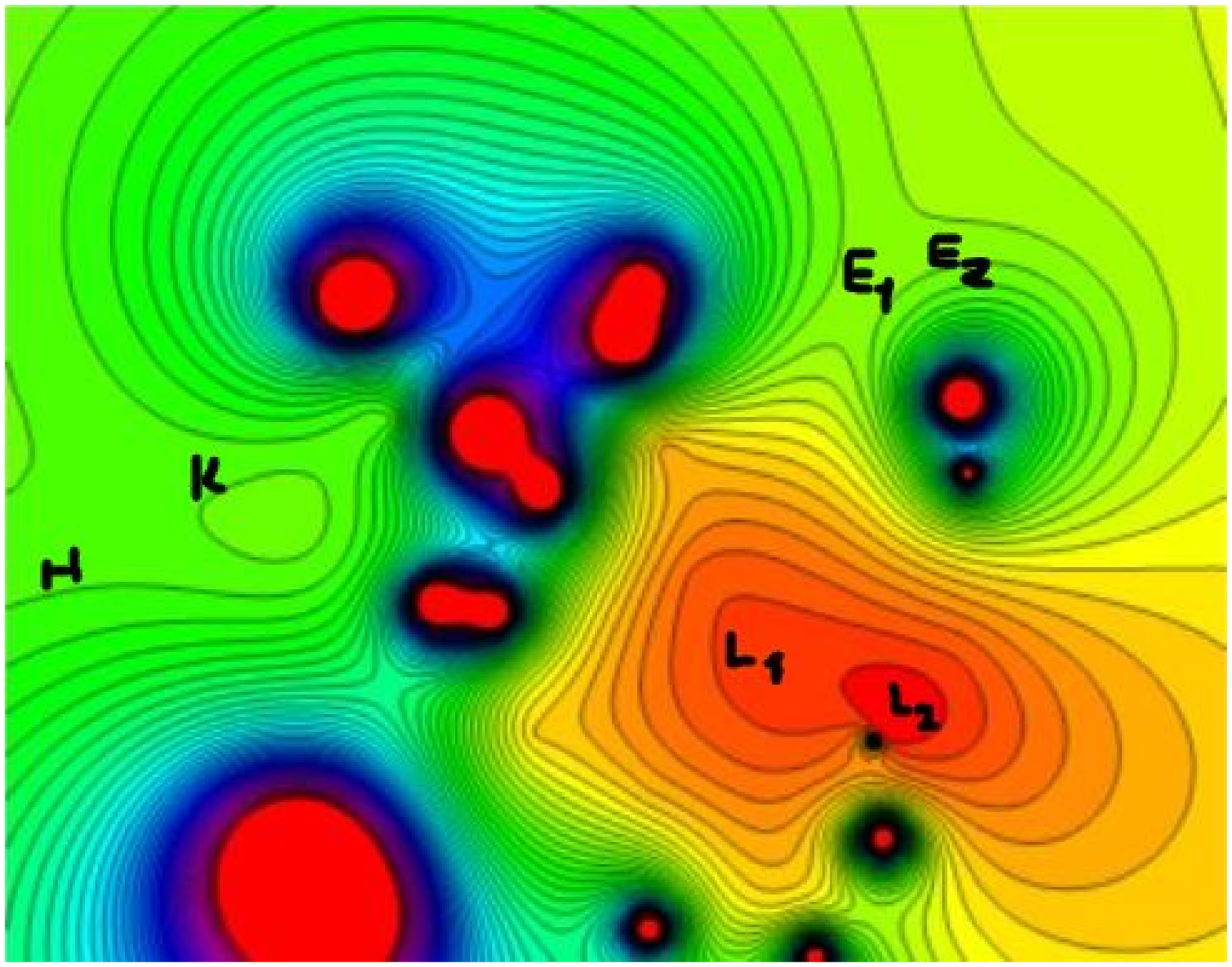}
\caption[]{Antipodal sets $H,K,E_1,E_2,L_1,L_2$ that are strings}
\label{fig:antipodesHKL}
\end{figure}

\subsection{Antipodal sets}
In considering a string-based BUT, we consider antipodal sets instead of antipodal points.  A subset $A$ of an $n$-dimensional real linear space $\mathbb{R}^n$ is a Petty \emph{antipodal set}, provided, for each pair of points $p,q\in A$, there exist disjoint parallel hyperplanes $P,Q$ such that $p\in P, q\in Q$~\cite[\S 2]{Petty1971antipodalSets}.  A \emph{hyperplane} is subspace $H = \left\{x\in \mathbb{R}^n: x\cdot v = c, v\in S^{n-1},c\in\mathbb{R}\right\}$ of a vector space~\cite{CoxMcKelvey1984hyperplane}.  In general, a hyperplane is any codimension-1 subspace of a vector space~\cite{Weisstein2016hyperplane}. 

\begin{figure}[!ht]
\centering
\subfigure[bounded worldsheet]
 {\label{fig:1worldsheet}\includegraphics[width=85mm]{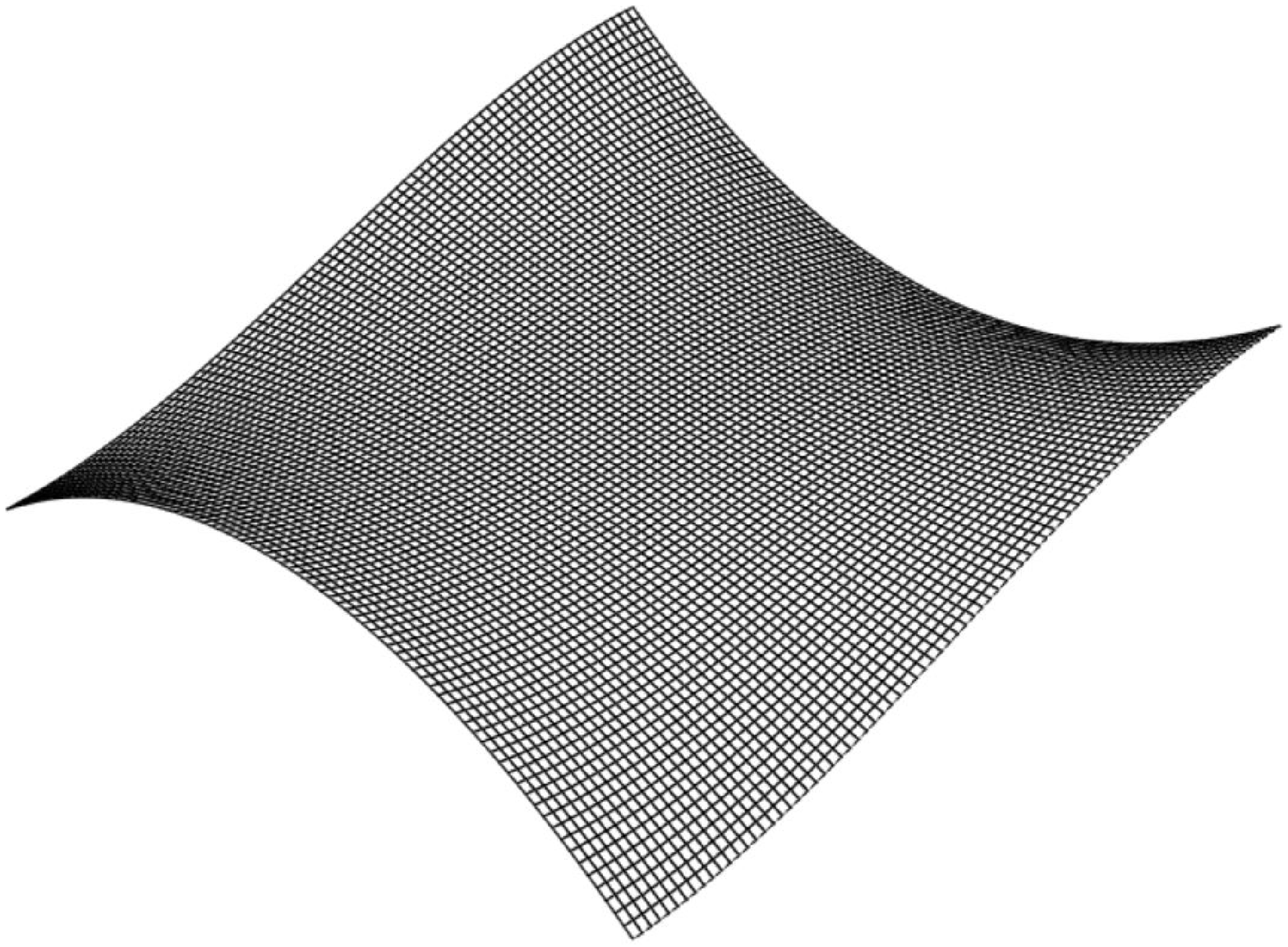}}\hfil
\subfigure[rolled worldsheet]
 {\label{fig:2worldsheet}\includegraphics[width=30mm]{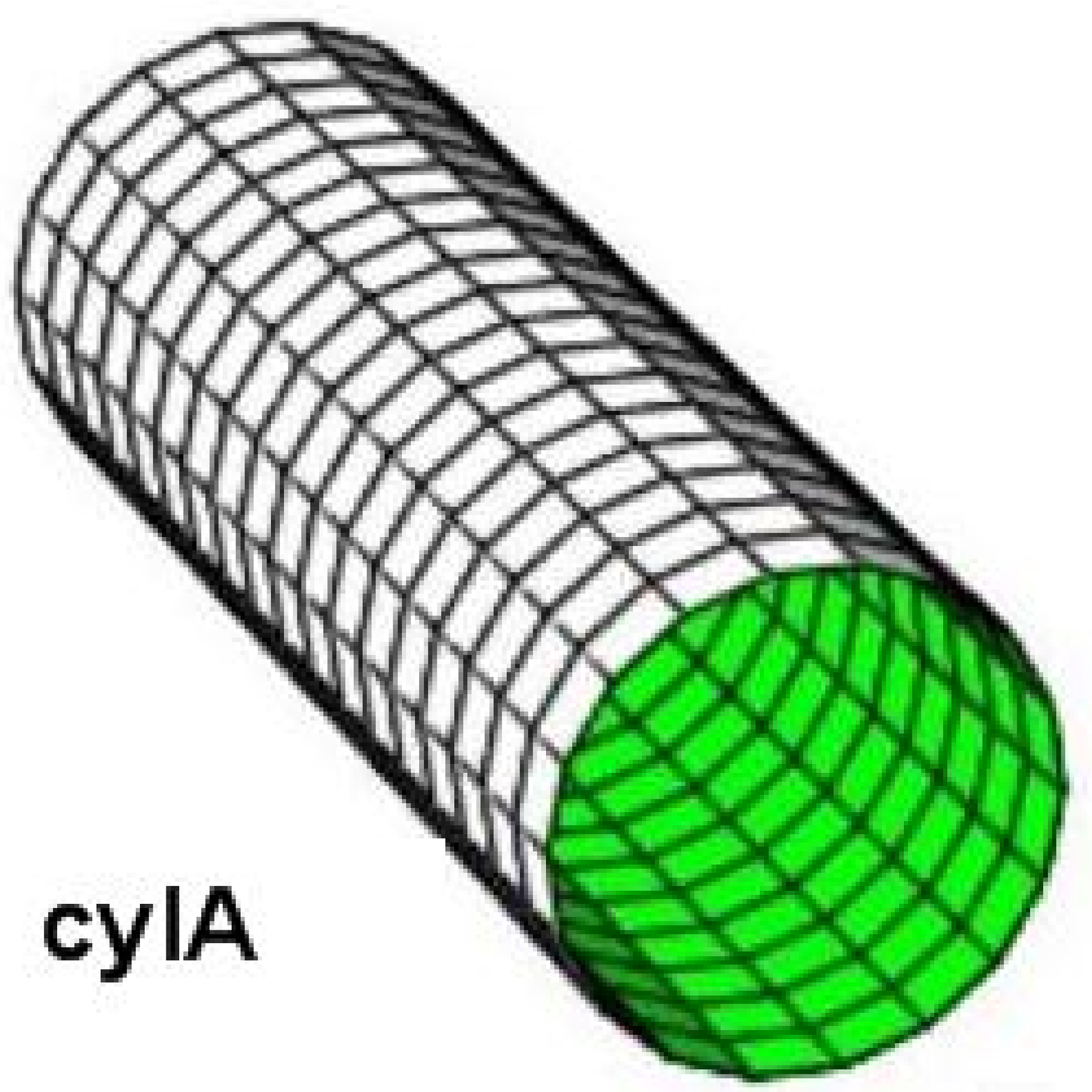}}
\caption[]{Bounded worldsheet $\mapsto$ cylinder surface}
\label{fig:corticalRegions}
\end{figure}
$\mbox{}$\\
\vspace{2mm}


Let $X$ be an $n$-dimensional Euclidean space $\mathbb{R}^n$ such that every member of $X$ is a string and let $S,\righthalfcap S\subset X$.  $A\in S,\righthalfcap A\in \righthalfcap S$ are \emph{antipodal sets}, provided $A,\righthalfcap A$ are subsets of disjoint parallel hyperplanes.  For example, strings are antipodal, provided the strings are subsets of disjoint parallel hyperplanes.  In this work, we relax the Petty antipodal set parallel hyperplane requirement.  That is, a pair of spatial regions (\emph{aka}, strings) $\str A,\righthalfcap \str A$ are \emph{antipodal}, provided the regions or strings contain proper substrings of disjoint parallel hyperplanes, {\em i.e.}, antipodes $A,\righthalfcap A$ contain proper substrings $\str A'\subset A,\righthalfcap \str A'\subset\righthalfcap A$ such that $\str A'\cap \righthalfcap \str A' = \emptyset$.  

\setlength{\intextsep}{0pt}
\begin{wrapfigure}[11]{R}{0.35\textwidth}
\begin{minipage}{3.2 cm}
\centering
\includegraphics[width=35mm]{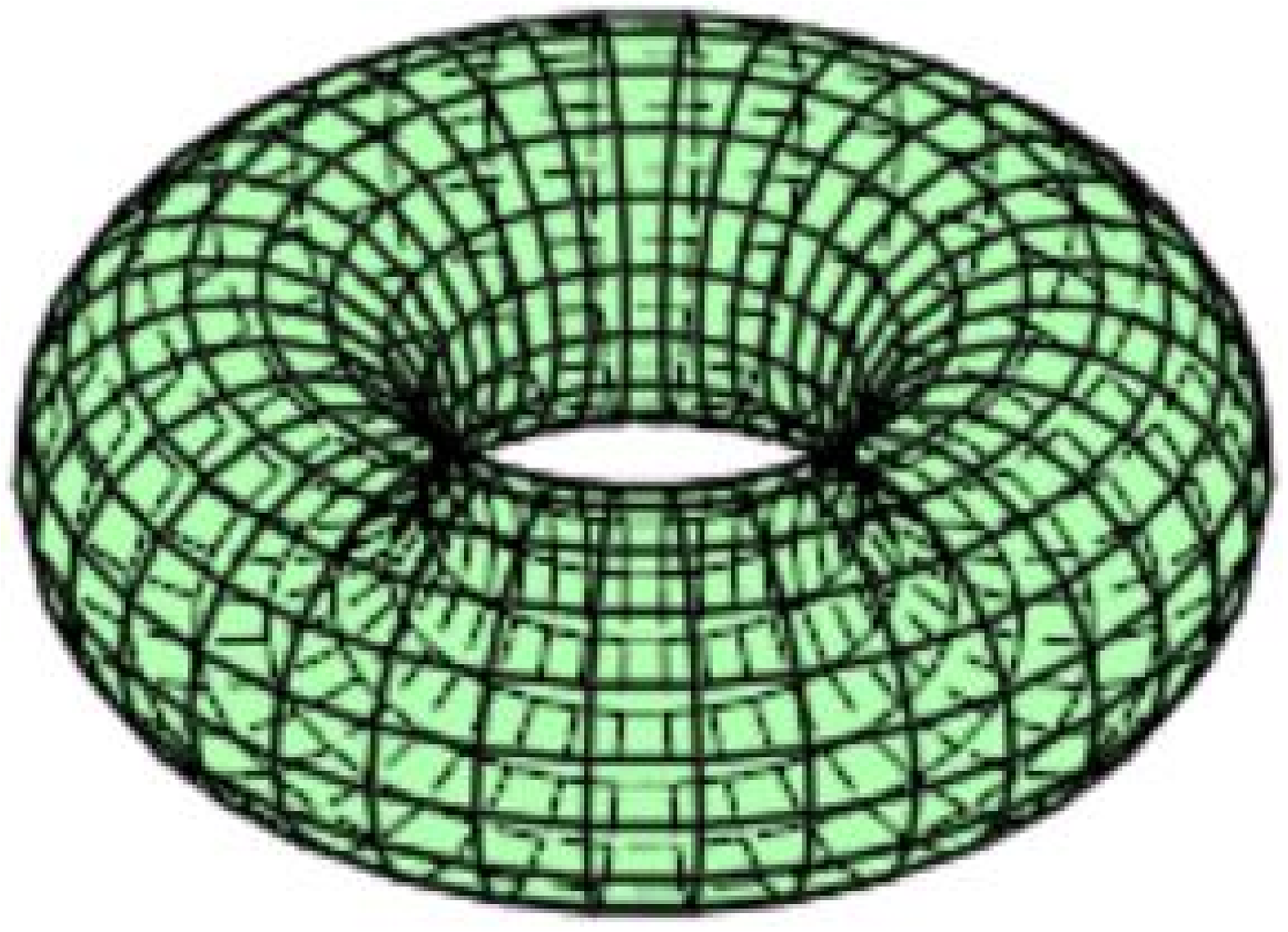}
\caption[]{\footnotesize \bf Torus}
\label{fig:torus}
\end{minipage}
\end{wrapfigure}

In other words, a pair of strings $\str A,\righthalfcap \str A$ are \emph{antipodal}, provided there are points $p\in \str A, q\in \righthalfcap \str A$ that belong to disjoint hyperplanes. 

$\mbox{}$\\
\vspace{3mm}
\begin{example}
Let $\str H,\str K,\str E_1,\str E_2,\str L_1,\str L_2$ in Fig.~\ref{fig:antipodesHKL} represent strings.  Then
\begin{align*}
\str H,\str K &\mbox{are antipodal, since}\ \str H\cap \str K = \emptyset,\\
\str H,\str E_1\cup \str E_2 &\mbox{are antipodal, since}\ \str H\cap \left(\str E_1\cup \str E_2\right) = \emptyset,\\
\str H,\str L_1\cup L_2 &\mbox{are antipodal, since}\ \str H\cap \left(\str L_1\cup \str L_2\right) = \emptyset,\\
\str E_1,\str L_1 &\mbox{are antipodal, since}\ \str E_1\cap \str L_1 = \emptyset,\\
\str E_1,\str E_2 &\mbox{are antipodal, since}\ X\cap Y \neq \emptyset,\ \mbox{for}\ X\in 2^{\str E_1},Y\in 2^{\str E_2},\\
\str L_1,\str L_2 &\mbox{are antipodal, since}\ X\cap Y \neq \emptyset,\ \mbox{for}\ X\in 2^{\str L_1},Y\in 2^{\str L_2}.
\end{align*}
By definition, $\str E_1,\str E_2$ and $\str L_1,\str L_2$ are antipodal, since
\begin{align*}
\left(\str E_1\cup \str E_2\setminus \str E_1\cap \str E_2\right) &\neq \emptyset,\ \mbox{and}\\
\left(\str L_1\cup \str L_2\setminus \str L_1\cap \str L_2\right) &\neq \emptyset.\mbox{\qquad \textcolor{blue}{$\blacksquare$}}
\end{align*}
\end{example} 

In a point-free geometry~\cite{DiConcilio2013mcs,Concilio2006PointFreeGeometry}, regions (nonempty sets) replace points as the primitives.  Let $M$ be a nonempty region of a space $X$.  Region $M$ is a \emph{worldsheet} (denoted by $\sh M$), provided every subregion of $\sh M$ contains at least one string.  In other words, a worldsheet is completely covered by strings.  Let $X$ be a collection of strings.  $X$ is a \emph{cover} for $\sh M$, provided $\sh M\subseteq X$.  Every member of a worldsheet is a string.

\begin{example}\label{ex:rolledUpWorldsheet}
\includegraphics[width=30mm]{1worldsheet}{\large $\boldsymbol{\longmapsto}$}\includegraphics[width=10mm]{2worldsheet} A worldsheet $\sh M$ with finite width and finite length is represented in Fig.~\ref{fig:1worldsheet}.  This worldsheet is rolled up to form the lateral surface of a cylinder represented in Fig.~\ref{fig:2worldsheet}, namely, $2\pi rh$ with radius $r$ and height $h$ equal to the length of $\sh M$.  We call this a worldsheet cylinder.  In effect, a flattened, bounded worldsheet maps to a worldsheet cylinder. \qquad \textcolor{blue}{$\blacksquare$}
\end{example}

The rolled up world sheet in Example~\ref{ex:rolledUpWorldsheet} is called a worldsheet cylinder surface.
Let $\sh M,\righthalfcap \sh M$ be a pair of worldsheets.  $\sh M,\righthalfcap \sh M$ are antipodal, provided there is at least one string $\str A\in \sh M, \str B\in \righthalfcap \sh M$ such that $\str A\cap \str B = \emptyset$.  A worldsheet cylinder maps to a worldsheet torus.

\begin{example}\label{ex:bendedUpWorldsheetTorus}
\includegraphics[width=10mm]{2worldsheet}\quad {\large $\boldsymbol{\longmapsto}$}\includegraphics[width=15mm]{torus} A worldsheet torus $\sh M$ with finite radius and finite length is represented in Fig.~\ref{fig:torus}.  This worldsheet torus is formed by bending a worldsheet cylinder until the ends meet.  In effect, a flattened, a worldsheet cylinder maps to worldsheet torus. \qquad \textcolor{blue}{$\blacksquare$}
\end{example}

\begin{conjecture}
A bounded worldsheet cylinder is homotopically equivalent to a worldsheet torus.
\qquad \textcolor{blue}{$\blacksquare$}
\end{conjecture}

\subsection{Strong Descriptive proximity}
The descriptive proximity $\delta_{\Phi}$ was introduced in~\cite{Peters2012ams}.   Let $A,B \subset X$ and let $\Phi(x)$ be a feature vector for $x\in X$, a nonempty set of non-abstract points such as picture points.  $A\ \delta_{\Phi}\ B$ reads $A$ is descriptively near $B$, provided $\Phi(x) = \Phi(y)$ for at least one pair of points, $x\in A, y\in B$.  From this, we obtain the description of a set and the descriptive intersection~\cite[\S 4.3, p. 84]{Naimpally2013} of $A$ and $B$ (denoted by $A\ \dcap\ B$) defined by
\begin{description}
\item[{\rm\bf ($\boldsymbol{\Phi}$)}] $\Phi(A) = \left\{\Phi(x)\in\mathbb{R}^n: x\in A\right\}$, set of feature vectors.
\item[{\rm\bf ($\boldsymbol{\dcap}$)}]  $A\ \dcap\ B = \left\{x\in A\cup B: \Phi(x)\in \Phi(A)\ \mbox{and}\ \Phi(x)\in \Phi(B)\right\}$.
\qquad \textcolor{blue}{$\blacksquare$}
\end{description}
Then swapping out $\near$ with $\dnear$ in each of the Lodato axioms defines a descriptive Lodato proximity. 

That is, a \textit{descriptive Lodato proximity $\dnear$} is a relation on the family of sets $2^X$, which satisfies the following axioms for all subsets $A, B, C $ of $X$.\\

\begin{description}
\item[{\rm\bf (dP0)}] $\emptyset\ \dfar\ A, \forall A \subset X $.
\item[{\rm\bf (dP1)}] $A\ \dnear\ B \Leftrightarrow B\ \dnear\ A$.
\item[{\rm\bf (dP2)}] $A\ \dcap\ B \neq \emptyset \Rightarrow\ A\ \dnear\ B$.
\item[{\rm\bf (dP3)}] $A\ \dnear\ (B \cup C) \Leftrightarrow A\ \dnear\ B $ or $A\ \dnear\ C$.
\item[{\rm\bf (dP4)}] $A\ \dnear\ B$ and $\{b\}\ \dnear\ C$ for each $b \in B \ \Rightarrow A\ \dnear\ C$. \qquad \textcolor{blue}{$\blacksquare$}
\end{description}
\vspace{3mm}
$\emptyset \dfar A$ reads the empty set is descriptively far from $A$.  Further $\dnear$ is \textit{descriptively separated }, if 
\vspace{3mm}
\begin{description}
\item[{\rm\bf (dP5)}] $\{x\}\ \dnear\ \{y\} \Rightarrow \Phi(x) = \Phi(y)$ ($x$ and $y$ have matching descriptions). \qquad \textcolor{blue}{$\blacksquare$}
\end{description}
\vspace{3mm}


\begin{proposition}\label{prop:dnear}~\cite{Peters2016arXivNerves}
Let $\left(X,\dnear\right)$ be a descriptive proximity space, $A,B\subset X$.  Then $A\ \dnear\ B \Rightarrow A\ \dcap\ B\neq \emptyset$.
\end{proposition}



Nonempty sets $A,B$ in a topological space $X$ equipped with the relation $\sn$, are \emph{strongly near} [\emph{strongly contacted}] (denoted $A\ \sn\ B$), provided the sets have at least one point in common.   The strong contact relation $\sn$ was introduced in~\cite{Peters2015JangjeonMSstrongProximity} and axiomatized in~\cite{PetersGuadagni2015stronglyNear},~\cite[\S 6 Appendix]{Guadagni2015thesis}.

Let $X$ be a topological space, $A, B, C \subset X$ and $x \in X$.  The relation $\sn$ on the family of subsets $2^X$ is a \emph{strong proximity}, provided it satisfies the following axioms.

\begin{description}
\item[{\rm\bf (snN0)}] $\emptyset\ \sfar\ A, \forall A \subset X $, and \ $X\ \sn\ A, \forall A \subset X$.
\item[{\rm\bf (snN1)}] $A \sn B \Leftrightarrow B\ \sn\ A$.
\item[{\rm\bf (snN2)}] $A\ \sn\ B$ implies $A\ \cap\ B\neq \emptyset$. 
\item[{\rm\bf (snN3)}] If $\{B_i\}_{i \in I}$ is an arbitrary family of subsets of $X$ and  $A\ \sn\ B_{i^*}$ for some $i^* \in I \ $ such that $\Int(B_{i^*})\neq \emptyset$, then $A\ \sn\ (\bigcup_{i \in I} B_i)$ 
\item[{\rm\bf (snN4)}]  $\mbox{int}A\ \cap\ \mbox{int} B \neq \emptyset \Rightarrow A\ \sn\ B$.  
\qquad \textcolor{blue}{$\blacksquare$}
\end{description}

\noindent When we write $A\ \sn\ B$, we read $A$ is \emph{strongly near} $B$ ($A$ \emph{strongly contacts} $B$).  The notation $A\ \notfar\ B$ reads $A$ is not strongly near $B$ ($A$ does not \emph{strongly contact} $B$). For each \emph{strong proximity} (\emph{strong contact}), we assume the following relations:
\begin{description}
\item[{\rm\bf (snN5)}] $x \in \Int (A) \Rightarrow x\ \sn\ A$ 
\item[{\rm\bf (snN6)}] $\{x\}\ \sn\ \{y\}\ \Leftrightarrow x=y$  \qquad \textcolor{blue}{$\blacksquare$} 
\end{description}

For strong proximity of the nonempty intersection of interiors, we have that $A \sn B \Leftrightarrow \Int A \cap \Int B \neq \emptyset$ or either $A$ or $B$ is equal to $X$, provided $A$ and $B$ are not singletons; if $A = \{x\}$, then $x \in \Int(B)$, and if $B$ too is a singleton, then $x=y$. It turns out that if $A \subset X$ is an open set, then each point that belongs to $A$ is strongly near $A$.  The bottom line is that strongly near sets always share points, which is another way of saying that sets with strong contact have nonempty intersection.\\

The descriptive strong proximity $\snd$ is the descriptive counterpart of $\sn$.
To obtain a \emph{descriptive strong Lodato proximity} (denoted by \emph{\bf dsn}), we swap out $\dnear$ in each of the descriptive Lodato axioms with the descriptive strong proximity $\snd$.  

Let $X$ be a topological space, $A, B, C \subset X$ and $x \in X$.  The relation $\snd$ on the family of subsets $2^X$ is a \emph{descriptive strong Lodato proximity}, provided it satisfies the following axioms.
\vspace{2mm}
\begin{description}
\item[{\rm\bf (dsnP0)}] $\emptyset\ {\sdfar}\ A, \forall A \subset X $, and \ $X\ \snd\ A, \forall A \subset X$.
\item[{\rm\bf (dsnP1)}] $A\ \snd\ B \Leftrightarrow B\ \snd\ A$.
\item[{\rm\bf (dsnP2)}] $A\ \snd\ B$ implies $A\ \dcap\ B\neq \emptyset$.  
\item[{\rm\bf (dsnP4)}] $\mbox{int}A\ \dcap\ \mbox{int} B \neq \emptyset \Rightarrow A\ \snd\ B$.  
\qquad \textcolor{blue}{$\blacksquare$}
\end{description}

\noindent When we write $A\ \snd\ B$, we read $A$ is \emph{descriptively strongly near} $B$.
For each \emph{descriptive strong proximity}, we assume the following relations:
\begin{description}
\item[{\rm\bf (dsnP5)}] $\Phi(x) \in \Phi(\Int (A)) \Rightarrow x\ \snd\ A$. 
\item[{\rm\bf (dsnP6)}] $\{x\}\ \snd\ \{y\} \Leftrightarrow \Phi(x) = \Phi(y)$.  
\qquad \textcolor{blue}{$\blacksquare$} 
\end{description}

\begin{definition}
Suppose that $(X, \tau_X, {\sn}_X) $ and $(Y, \tau_Y, {\sn}_Y)$ are topological spaces endowed with strong proximities~\cite{PetersGuadagni2015strongConnectedness}.  We say that the map $f: X \rightarrow Y$ is \emph{strongly proximal continuous} and we write \emph{\bf s.p.c.} if and only if, for $A, B \subset X$, 
\[
\ A\ {\sn}_X\ B \Rightarrow f(A)\ {\sn}_Y\ f(B). \mbox{\qquad \textcolor{black}{$\blacksquare$}}
\] 
\end{definition}

\begin{theorem}\label{open}~{\rm \cite{PetersGuadagni2015strongConnectedness}}
Suppose that $(X, \tau_X, {\sn}_X) $ and $(Y, \tau_Y, {\sn}_Y)$ are topological spaces endowed with compatible strong proximities and $f: X \rightarrow Y$ is s.p.c. Then $f$ is an open mapping, {\em i.e.}, $f$ maps open sets in open sets.
\end{theorem}

Let $\righthalfcap x$ be any point in $X\setminus \left\{x\right\}$ that is not $x\in X$.  The Borsuk-Ulam Theorem (BUT) has many different region-based incarnations.

\begin{theorem}\label{thm:Borsuk}~{\rm \cite{PetersGuadagni2015strongConnectedness}}
If $f:S^n\longrightarrow \mathbb{R}^n$ is $\sn$-continuous (s.p.c.), then $f(x) = f(\righthalfcap x)$ for some $x\in X$.
\end{theorem}
\begin{proof}
The mapping $f$ is s.p.c. if and only if $A\ {\sn}_{S^n}\ B$ implies $f(A)\ {\sn}_{\mathbb{R}^n}\ f(B)$.  Let $A = \left\{x\right\}, B = \left\{\righthalfcap x\right\}$ for some $x\in X, \righthalfcap x\in X\setminus \left\{x\right\}$.  From $\sn$-Axiom (snN6), $\left\{x\right\}\ \sn\ \left\{\righthalfcap x\right\}\Leftrightarrow \left\{x\right\} = \left\{\righthalfcap x\right\}$.  Hence,
$f(x) = f(\righthalfcap x)$.
\end{proof}

\setlength{\intextsep}{0pt}
\begin{wrapfigure}[11]{R}{0.35\textwidth}
\begin{minipage}{5.0 cm}
\centering
\includegraphics[width=40mm]{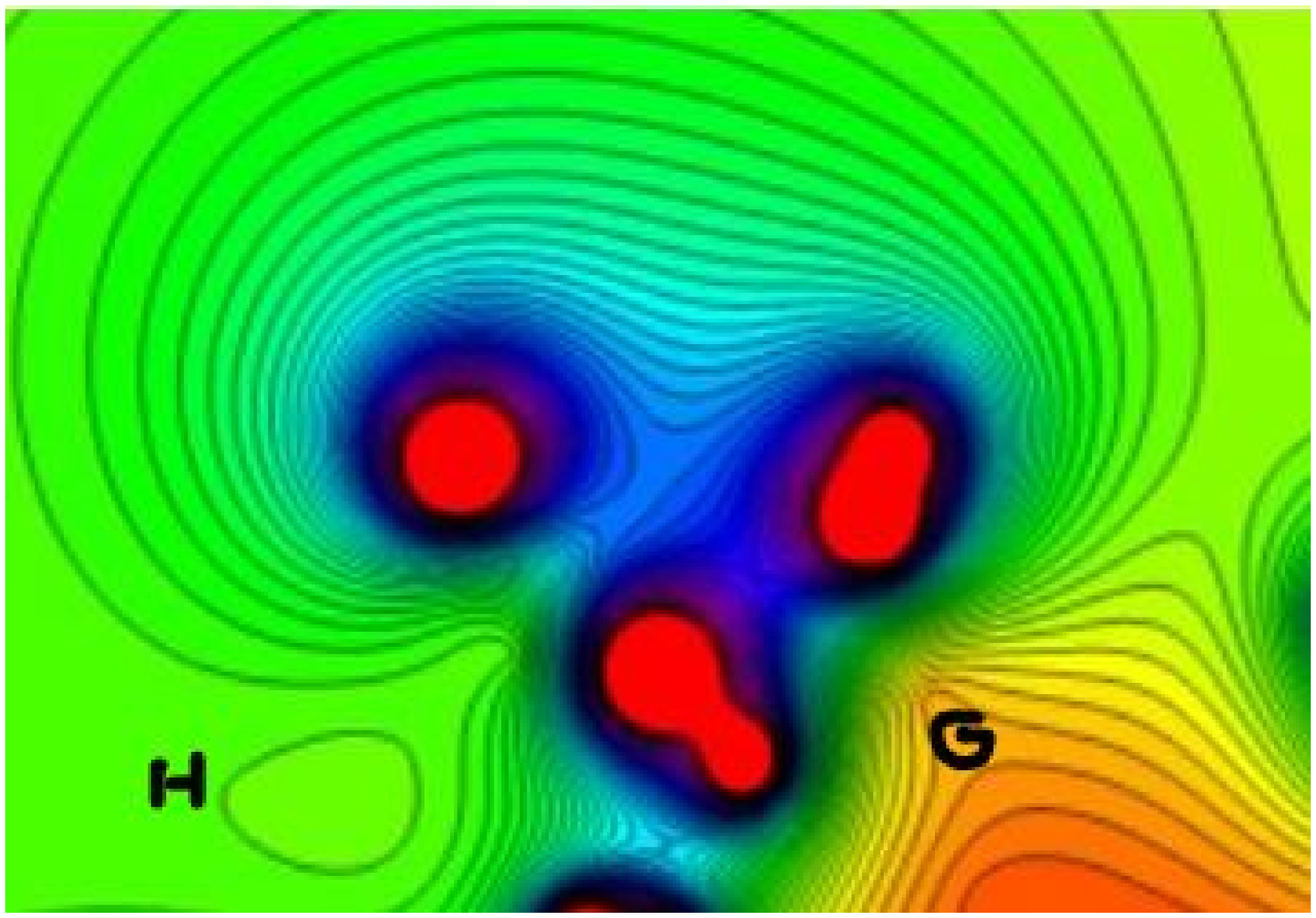}
\caption[]{Antipodal strings $H,G$}
\label{fig:worldlinesHG}
\end{minipage}
\end{wrapfigure}

\begin{corollary}
If $f:S^n\longrightarrow \mathbb{R}^n$ is $\sn$-continuous (s.p.c.), then $f(x) = f(\righthalfcap x)$ for some $x\in X$.
\end{corollary}

Since the proof of Theorem~\ref{thm:Borsuk} depends on the domain and range of mapping $f$ being compatible topological spaces equipped with a s.p.c. map and does not depend on the geometry of $S^n$, we have

\begin{theorem}\label{thm:BorsukX}~{\rm \cite{PetersGuadagni2015strongConnectedness}}
Let $(X, \tau_X, {\sn}_X) $ and $(\mathbb{R}^n, \tau_{\mathbb{R}^n}, {\sn}_{\mathbb{R}^n})$ be topological spaces endowed with compatible strong proximities.
If $f:X\longrightarrow \mathbb{R}^n$ is $\sn$-continuous (s.p.c.), then $f(x) = f(\righthalfcap x)$ for some $x\in X$.
\end{theorem}

\begin{corollary}
Let $X$ be a nonempty set of strings.  If $f:S^n\longrightarrow \mathbb{R}^n$ is $\sn$-continuous (s.p.c.), then $f(x) = f(\righthalfcap x)$ for some $x\in X$.
\end{corollary}

\section{String-Based Borsuk Ulam Theorem (strBUT)}
This section considers string-based forms of the Borsuk Ulam Theorem (strBUT).  
Recall that a \emph{region} is limited to a set in a metric topological space.   By definition, a string on the surface of an $n$-sphere is a line that represents the path traced by moving particle along the surface of the $S^n$.  Disjoint strings on the surface of $S^n$ that are antipodal and with matching description are descriptively near \emph{antipodal strings}.    A pair of strings $A,\righthalfcap A$ are \emph{antipodal}, provided, for some $p\in A, q\in \righthalfcap A$, there exist disjoint parallel hyperplanes $P,Q\subset S^n$ such that $p\in P$ and $q\in Q$.   Such strings can be spatially far apart and also descriptively near.

\begin{example}\label{ex:antipodalStrings}
A pair of antipodal strings are represented by $H,G$ in Fig.~\ref{fig:worldlinesHG}.  Strings $H,G$ are antipodal, since they have no points in common.  Let bounded shape be a feature of a string.  A shape is bounded, provided the shape is surrounded (contained in) another shape. Then $H,G$ are descriptively near, since they are both bounded shapes.
\qquad \textcolor{blue}{$\blacksquare$}
\end{example}

We are interested in the case where strongly near strings are mapped to strongly near strings in a feature space.  To arrive at a string-based form of Theorem~\ref{thm:BorsukX}, we introduce region-based strong proximal continuity.  Let $X$ be a nonempty set and let $2^X$ denote the family of all subsets in $X$.  For example, $2^{S^n}$ is the family of all subsets on the surface of an $n$-sphere.
$\mbox{}$\\
\vspace{3mm}

\begin{definition}{\rm ~\cite[\S 5.7]{Peters2016ISRLcomputationalProximity}}.\\
Let $X,Y$ be nonempty sets.  Suppose that $(2^X, \tau_{2^X}, \dnear)$ and $(Y, \tau_{Y}, \dnear)$ are topological spaces endowed with strong proximities.  We say that $f: 2^X \rightarrow Y$ is \emph{region strongly proximal continuous} and we write \emph{\bf Re.d.p.c.} if and only if, for $A, B \in 2^X$, 
\[
\ A\ \dnear\ B \Rightarrow f(A)\ \dnear\ f(B). \mbox{\qquad \textcolor{blue}{$\blacksquare$}}
\] 
\end{definition} 

\begin{lemma}\label{lemma:redBorsuk}~{\rm \cite{PetersGuadagni2015strongConnectedness}}
Suppose that $(2^{\mathbb{R}^n}, \tau_{2^{\mathbb{R}^n}}, \dnear) $ and $(\mathbb{R}^n, \tau_{\mathbb{R}^n}, \dnear)$ are topological spaces endowed with compatible descriptive proximities and $f$ is a $\dnear$ Re.d.p.c. continuous mapping on the family of regions $2^{\mathbb{R}^n}$ into $\mathbb{R}^n$.
If $f(A)\in \mathbb{R}^n$ is a description common to antipodal regions $A,\righthalfcap A\in 2^{\mathbb{R}^n}\setminus A$, then $f(A) = f(\righthalfcap A)$ for some $\righthalfcap A$.
\end{lemma}

Lemma~\ref{lemma:redBorsuk} is restricted to regions in $2^{\mathbb{R}^n}$ described by feature vectors in an $n$-dimensional Euclidean space $\mathbb{R}^n$.  Next, consider regions on the surface of an $n$-sphere $S^n$.  Each feature vector $f(A)$ in $\mathbb{R}^n$ describes a region $A\in 2^{S^n}$.  Then we obtain the following result.

\begin{theorem}\label{thm:redSnBUT}~{\rm \cite{PetersGuadagni2015strongConnectedness}}
Suppose that $(2^{S^n}, \tau_{2^{S^n}}, \snd)$ and $(\mathbb{R}^n, \tau_{\mathbb{R}^n}, \snd)$ are topological spaces endowed with compatible strong proximities.  Let $A\in 2^{S^n}$, a region in the family of regions in $2^{S^n}$.
If $f:2^{S^n}\longrightarrow \mathbb{R}^n$ is $\dnear$ Re.d.p.c. continuous, then $f(A) = f(\righthalfcap A)$ for antipodal region $\righthalfcap A\in 2^{S^n}$.
\end{theorem}

\begin{theorem}\label{cor:strBUTsimplest}
Suppose that $(2^{S^n}, \tau_{2^{S^n}}, \snd)$ and $(\mathbb{R}^n, \tau_{\mathbb{R}^n}, \snd)$ are topological spaces endowed with compatible strong proximities.  Let $A\in 2^{S^n}$, a string in the family of strings in $2^{S^n}$.
If $f:2^{S^n}\longrightarrow \mathbb{R}^n$ is $\dnear$ Re.d.p.c. continuous, then $f(A) = f(\righthalfcap A)$ for antipodal string $\righthalfcap A\in 2^{S^n}$.
\end{theorem}
\begin{proof}
Let each string be a spatial subregion of $2^{S^n}$.  Let $\str A, str(\righthalfcap A)\in 2^{S^n}$.  Swap out $A,\righthalfcap A\in 2^{\mathbb{R}^n}$ with $\str A, str(\righthalfcap A)\in 2^{S^n}$ in Theorem~\ref{thm:redSnBUT} and the result follows.
\end{proof} 

\begin{theorem}\label{cor:strBUTsheet}
Suppose that $(2^{S^n}, \tau_{2^{S^n}}, \snd)$ and $(\mathbb{R}^n, \tau_{\mathbb{R}^n}, \snd)$ are topological spaces endowed with compatible strong proximities.  Let $\sh M\in 2^{S^n}$ be a worldsheet in the family of worldsheets in $2^{S^n}$.
If $f:2^{S^n}\longrightarrow \mathbb{R}^n$ is $\dnear$ Re.d.p.c. continuous, then $f(\sh M) = f(\righthalfcap \sh M)$ for antipodal $\righthalfcap \sh M\in 2^{S^n}$.
\end{theorem}
\begin{proof}
The proof is symmetric with the proof of Theorem~\ref{cor:strBUTsimplest}.
\end{proof} 

\begin{remark}
Theorem~\ref{cor:strBUTsimplest} and Theorem~\ref{cor:strBUTsheet} are the simplest forms of string-based Borsuk-Ulam Theorem (strBUT).  In this section, we also consider other forms of strBUT that arise naturally from strong forms of descriptive proximity and which have proved to be useful in a number of applications.
\qquad \textcolor{blue}{$\blacksquare$}
\end{remark}

\begin{definition}{\bf Region-Based $\sn$-Continuous Mapping}~\cite[\S 5.7]{Peters2016ISRLcomputationalProximity}.\\
Let $X,Y$ be nonempty sets.  Suppose that $(2^X, \tau_{2^X}, {\sn}_{2^X}) $ and $(Y, \tau_{Y}, {\sn}_{Y})$ are topological spaces endowed with strong proximities.  We say that $f: 2^X \rightarrow Y$ is \emph{region strongly proximal continuous} and we write \emph{\bf Re.s.p.c.} if and only if, for $A, B \in 2^X$, 
\[
\ A\ {\sn}_X\ B \Rightarrow f(A)\ {\sn}_Y\ f(B). \mbox{\qquad \textcolor{blue}{$\blacksquare$}}
\] 
\end{definition} 

For an introduction to s.p.c. mappings, see~\cite{PetersGuadagni2016spcMappings}.  Let $A\in 2^{S^n},\righthalfcap A\in 2^{S^n}\setminus A$.  For a Re.s.p.c. mapping $f: 2^{S^n} \rightarrow \mathbb{R}^n$ on the collection of subsets $2^{S^n}$ into $\mathbb{R}^n$, the assumption is that $2^{S^n}$ is a region-based object space (each object is represented by a nonempty region) and $\mathbb{R}^n$ is a feature space (each region $A$ in $2^{S^n}$ maps to a feature vector $y$ in an $n$-dimensional Euclidean space $\mathbb{R}^n$ such that $y\in \mathbb{R}^n$ is a description of region $A$ that matches the description of $\righthalfcap A$). 

\begin{definition}{\bf Region-Based $\snd$-Continuous Mapping}.\\
The mapping $f: 2^X \rightarrow \mathbb{R}^n$ is \emph{region $\snd$-continuous} and we write \emph{\bf Re.d.s.p.c.} if and only if, for $A, B \in 2^X$, 
\[
\ A\ \snd\ B \Rightarrow f(A)\ \snd\ f(B),
\]
where $f(A)\in \mathbb{R}^n$ is a feature vector that describes region $A$.
\mbox{\qquad \textcolor{blue}{$\blacksquare$}}
\end{definition}

\begin{lemma}\label{lemma:reBorsuk} {\rm \cite{PetersTozzi2016reBUT}}.\\
Suppose that $(2^{\mathbb{R}^n}, \tau_{2^{\mathbb{R}^n}}, \snd) $ and $(\mathbb{R}^n, \tau_{\mathbb{R}^n}, \snd)$ are topological spaces endowed with compatible strong descriptive proximities and $f$ is a $\snd$ continuous mapping on the family of regions $2^{\mathbb{R}^n}$ into $\mathbb{R}^n$.
If $f(A)\in \mathbb{R}^n$ is a description common to antipodal regions $A,\righthalfcap A\in 2^{\mathbb{R}^n}\setminus A$, then $f(A) = f(\righthalfcap A)$ for some $\righthalfcap A$.
\end{lemma}

\setlength{\intextsep}{0pt}
\begin{wrapfigure}[12]{R}{0.35\textwidth}
\begin{minipage}{3.2 cm}
\centering
\includegraphics[width=27mm]{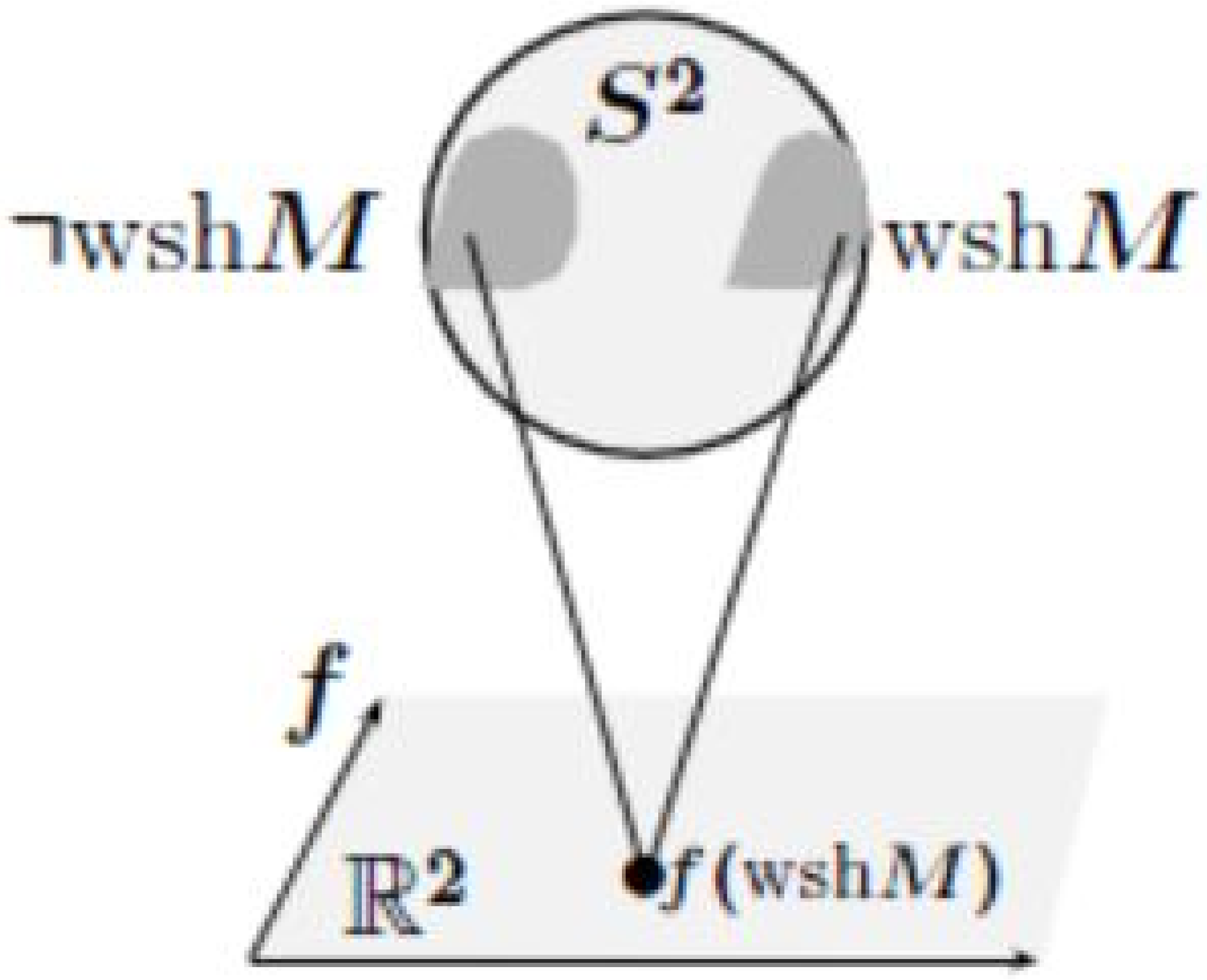}
\caption[]{\footnotesize $\boldsymbol{2^{S^2}\mapsto \mathbb{R}^2}$}
\label{fig:re2BUT}
\end{minipage}
\end{wrapfigure}
$\mbox{}$\\
\vspace{3mm}

Lemma~\ref{lemma:reBorsuk} is restricted to regions in $(2^{\mathbb{R}^n}$ described by feature vectors in an $n$-dimensional Euclidean space $\mathbb{R}^n$.  Next, consider regions on the surface of an $n$-sphere $S^n$.  Each feature vector $f(A)$ in $\mathbb{R}^n$ describes a region $A\in 2^{S^n}$.  Then we obtain the following result.

\begin{theorem}\label{thm:reSnBUT} {\rm \cite{PetersTozzi2016reBUT}}.\\
Suppose that $(2^{S^n}, \tau_{2^{S^n}}, \snd)$ and $(\mathbb{R}^n, \tau_{\mathbb{R}^n}, \snd)$ are topological spaces endowed with compatible strong proximities.  Let $A\in 2^{S^n}$, a region in the family of regions in $2^{S^n}$.
If $f:2^{S^n}\longrightarrow \mathbb{R}^n$ is $\sn$-Re.s.p continuous, then $f(A) = f(\righthalfcap A)$ for region $\righthalfcap A\in 2^{S^n}$.
\end{theorem}

\begin{theorem}\label{thm:reSnBUTworldsheet}
Suppose that $(2^{S^n}, \tau_{2^{S^n}}, \snd)$ and $(\mathbb{R}^n, \tau_{\mathbb{R}^n}, \snd)$ are topological spaces endowed with compatible strong proximities.  Let $\sh A\in 2^{S^n}$ be a worldsheet in the family of regions in $2^{S^n}$.
If $f:2^{S^n}\longrightarrow \mathbb{R}^n$ is $\sn$-Re.s.p continuous, then $f(\sh A) = f(\righthalfcap \sh A)$ for region $\righthalfcap \sh A\in 2^{S^n}$.
\end{theorem}
\begin{proof}
Since a worldsheet $\sh A\in 2^{S^n}$ is a region in $2^{S^n}$, the result follows from Theorem~\ref{thm:reSnBUT}.
\end{proof}

\begin{example}
The $\sn$-Re.s.p continuous mapping $f:2^{S^2}\longrightarrow \mathbb{R}^2$ is represented in Fig.~\ref{fig:re2BUT}.  In this example, it is assumed that the worldsheets $\sh M,\righthalfcap \sh M$ have matching descriptions one feature, {\em e.g.}, area.  In that case, $f(\sh M) = f(\righthalfcap \sh M)$.
\qquad \textcolor{blue}{$\blacksquare$} 
\end{example}

In the proof of Theorem~\ref{thm:reSnBUT} and Theorem~\ref{thm:reSnBUTworldsheet}, we do not depend on the fact that each region is on the surface of a hypersphere $S^n$.  For this reason, we are at liberty to introduce a more general region-based Borsuk-Ulam Theorem (denoted by reBUT), applicable to strings and worldsheets.

Let $A\in 2^X,\righthalfcap A\in 2^X\setminus A$.  For a Re.s.p.c. mapping $f: 2^X \rightarrow \mathbb{R}^n$ on the collection of subsets $2^X$ to $\mathbb{R}^n$, the assumption is that $2^X$ is a region-based object space (each object is represented by a nonempty region) and $\mathbb{R}^n$ is a feature space (each region $A$ in $2^X$ maps to a feature vector $y$ in $\mathbb{R}^n$ such that $y$ is a description of region $A$).  Then we obtain the following result

\begin{theorem}\label{thm:reRnBUT} {\rm \cite{PetersTozzi2016reBUT}}.\\
Suppose that $(X, \tau_X, {\sn}_X) $ and $(\mathbb{R}^n, \tau_{\mathbb{R}^n}, {\sn}_{\mathbb{R}^n})$ are topological spaces endowed with compatible strong proximities.
If $f:2^X\longrightarrow \mathbb{R}^n$ is $\sn$-Re.s.p continuous, then $f(A) = f(\righthalfcap A)$ for some $A\in 2^X$.
\end{theorem}

\begin{theorem}\label{thm:reRnBUTsheet} {\bf $\boldsymbol{\strBUT}$ for Worldsheets}.\\
Suppose that $(X, \tau_X, {\sn}_X) $ and $(\mathbb{R}^n, \tau_{\mathbb{R}^n}, {\sn}_{\mathbb{R}^n})$ are topological spaces endowed with compatible strong proximities, $\sh M, \righthalfcap \sh M\subset X$.
If $f:2^X\longrightarrow \mathbb{R}^n$ is $\sn$-Re.s.p continuous, then $f(\sh M) = f(\righthalfcap \sh M)$ for some $\sheet M\in 2^X$.
\end{theorem}
\begin{proof}
Let each worldsheet be a spatial subregion of $X$.  Let $\sh M, \righthalfcap \sh M\subset X$.  Swap out $A,\righthalfcap A\in 2^{\mathbb{R}^n}$ with $\sh M, \righthalfcap \sh M$ in Theorem~\ref{thm:reRnBUT} and the result follows.
\end{proof} 

\begin{example}
Let $\sh M, \righthalfcap \sh M$ in Fig.~\ref{fig:re2BUT} represent a pair of antipodal worldsheets.  For simplicity, we consider only the feature worldsheet area.  Let $f:2^X\longrightarrow \mathbb{R}^n$ map a worldsheet $\sh M$ to a real number that is the area of $\sh M$, \emph{i.e.}, $f(\sh M)\in \mathbb{R}^2$ equals the area of $\sh M$.  It is clear that more than one antipodal worldsheet will have the same area.  Then, from Theorem~\ref{thm:reRnBUTsheet}, $f(\sh M) = f(\righthalfcap \sh M)$ for some $\righthalfcap \sh M\in 2^X$.
\qquad \textcolor{blue}{$\blacksquare$} 
\end{example}

\begin{lemma}\label{lemma:rexBorsuk} {\bf Region Descriptions in a $k$-Dimensional Space}.\\
Suppose that $(2^{S^n}, \tau_{2^{S^n}}, {\sn}_{2^{S^n}}) $ and $(\mathbb{R}^k, \tau_{\mathbb{R}^k}, {\sn}_{\mathbb{R}^k})$ are topological spaces endowed with compatible strong proximities and $f$ on the family of regions $2^{S^n}$ maps into $\mathbb{R}^k, k > 0$.
If $f(A)\in \mathbb{R}^k$ is a description common to antipodal regions $\righthalfcap A\in 2^{S^n}\setminus A$, then $f(A) = f(\righthalfcap A)$ in $\mathbb{R}^k$ for some $\righthalfcap A$ in $2^{S^n}$ for $f:2^{S^n}\longrightarrow \mathbb{R}^k$.
\end{lemma}
\begin{proof}
Let $k > 0$.  The assumption is that $f(A)$ is a feature vector in a $k$-dimensional feature space that describes $A\in 2^{S^n}$ as well as at least one other region $\righthalfcap A\in 2^{S^n}\setminus A$ in $2^{S^n}$.  We consider only the case for $k = 1$ for shape-connected regions that are disks in a finite, bounded, rectangular shaped space in the Euclidean plane, where every region has 4,3 or 2 adjacent disks, {\em e.g.}, each corner region has at most 2 adjacent disks.
Let
\[
f(A) =
 \begin{cases}
 2, &\text{if $A$ is a corner region with 2 adjacent polygons},\\
 level > 2, &\text{otherwise}.
 \end{cases}
\]
Let $A,\righthalfcap A$ be antipodal corner regions.  Then $f(A) = f(\righthalfcap A)$ in $\mathbb{R}^1$. 
\end{proof}

Lemma~\ref{lemma:rexBorsuk} leads to a version of reBUT for region descriptions in a $k$-dimensional feature space.

\begin{theorem}\label{thm:rexBUT} {\bf Proximal Region-Based Borsuk-Ulam Theorem (rexBUT)}.\\
Suppose that $(X, \tau_X, {\sn}_X) $ and $(\mathbb{R}^k, \tau_{\mathbb{R}^k}, {\sn}_{\mathbb{R}^k}), k > 0$ are topological spaces endowed with compatible strong proximities.
If $f:2^X\longrightarrow \mathbb{R}^k$ is $\sn$-Re.s.p continuous, then $f(A) = f(\righthalfcap A)$ for some $A\in 2^X$.
\end{theorem}
\begin{proof}
Swap out $f:2^X\longrightarrow \mathbb{R}^n$ with $f:2^X\longrightarrow \mathbb{R}^k, k > 0$ in the proof of Theorem~\ref{thm:reRnBUT} and the result follows.
\end{proof}

A \emph{string space} is a nonempty set of strings.  A \emph{worldsheet space} is a nonempty set of worldsheets.

\begin{corollary}
Let $X$ be a string space.  Assume $(X, \tau_X, {\sn}_X)$, $(\mathbb{R}^k, \tau_{\mathbb{R}^k}, {\sn}_{\mathbb{R}^k})$, $k > 0$ are topological string spaces endowed with compatible strong proximities.
If $f:2^X\longrightarrow \mathbb{R}^k$ is $\sn$-Re.s.p continuous, then $f(\str A) = f(\righthalfcap \str A)$ for some  string $\str A\in 2^X$.
\end{corollary}

\begin{corollary}
Let $X$ be a worldsheet space.  Assume $(X, \tau_X, {\sn}_X) $, $(\mathbb{R}^k, \tau_{\mathbb{R}^k}, {\sn}_{\mathbb{R}^k})$, $k > 0$ are topological worldsheet spaces endowed with compatible strong proximities.
If $f:2^X\longrightarrow \mathbb{R}^k$ is $\sn$-Re.s.p continuous, then $f(\sh A) = f(\righthalfcap \sh A)$ for some worldsheet $\sh A\in 2^X$.
\end{corollary}

\subsection{String Space and Worldsheet Space}
A \emph{string space} is a nonempty set of strings.  A \emph{worldsheet space} is a nonempty set of worldsheets.

\begin{corollary}
Let $X$ be a string space.  Assume $(X, \tau_X, {\sn}_X)$, $(\mathbb{R}^k, \tau_{\mathbb{R}^k}, {\sn}_{\mathbb{R}^k})$, $k > 0$ are topological string spaces endowed with compatible strong proximities.
If $f:2^X\longrightarrow \mathbb{R}^k$ is $\sn$-Re.s.p continuous, then $f(\str A) = f(\righthalfcap \str A)$ for some  string $\str A\in 2^X$.
\end{corollary}

\begin{corollary}
Let $X$ be a worldsheet space.  Assume $(X, \tau_X, {\sn}_X) $, $(\mathbb{R}^k, \tau_{\mathbb{R}^k}, {\sn}_{\mathbb{R}^k})$, $k > 0$ are topological worldsheet spaces endowed with compatible strong proximities.
If $f:2^X\longrightarrow \mathbb{R}^k$ is $\sn$-Re.s.p continuous, then $f(\sh A) = f(\righthalfcap \sh A)$ for some worldsheet $\sh A\in 2^X$.
\end{corollary}

Let $A$ be a region in the family of sets $2^X$, $\Phi(A)$ a feature vector with $k$ components that describes region $A$.  A straightforward extension of Theorem~\ref{thm:rexBUT} leads to a continuous mapping of antipodal regions in an $n$-dimensional space in $X$ to regions in a $(k)$-dimensional feature space $\mathbb{R}^{k}$. 
\begin{theorem}\label{thm:re2reBUT} {\bf Region-2-Region Based Borsuk-Ulam Theorem (re2reBUT)}.\\
Suppose that $(X, \tau_X, {\sn}_X)$, where space $X$ is $n$-dimensional and $(\mathbb{R}^{k}, \tau_{\mathbb{R}^{k}}, {\sn}_{\mathbb{R}^{k}}), k > 0$ are topological spaces endowed with compatible strong proximities.
If $f:2^X\longrightarrow \mathbb{R}^{k}$ is $\sn$-Re.s.p continuous, then $f(A) = f(\righthalfcap A)$ for some $A\in 2^X$.
\end{theorem}
\begin{proof}
From Lemma~\ref{lemma:rexBorsuk} and swapping out $f:2^X\longrightarrow \mathbb{R}^n$ with $f:2^X\longrightarrow \mathbb{R}^{n+k}, k > 0$ in the proof of Theorem~\ref{thm:rexBUT}, the result follows.
\end{proof}


From Theorem~\ref{thm:re2reBUT}, we obtain the following results relative to strings and worldsheets.

\begin{corollary}\label{cor:stringSpace}
Let $X$ be a string space.  Assume $(X, \tau_X, {\sn}_X)$, $(\mathbb{R}^k, \tau_{\mathbb{R}^k}, {\sn}_{\mathbb{R}^k})$, $k > 0$ are topological string spaces endowed with compatible strong proximities.
If $f:2^X\longrightarrow \mathbb{R}^k$ is $\sn$-Re.s.p continuous, then $f(\str A) = f(\righthalfcap \str A)$ for some set of strings $\str A\in 2^X$.
\end{corollary}

\begin{corollary}
Let $X$ be a worldsheet space.  Assume $(X, \tau_X, {\sn}_X)$, $(\mathbb{R}^k, \tau_{\mathbb{R}^k}, {\sn}_{\mathbb{R}^k})$, $k > 0$ are topological worldsheet spaces endowed with compatible strong proximities.
If $f:2^X\longrightarrow \mathbb{R}^k$ is $\sn$-Re.s.p continuous, then $f(\sh A) = f(\righthalfcap \sh A)$ for some set of worldsheets $\sh A\in 2^X$.
\end{corollary}

\begin{example}
Let the Euclidean spaces $S^2$ and $\mathbb{R}^3$ be endowed with the strong proximity $\sn$ and let $\str A,\righthalfcap \str A$ be antipodal strings in $S^2$.   Further, let $f$ be a proximally continuous mapping on $2^{S^2}$ into $2^{\mathbb{R}^4}$ defined by
\begin{align*}
\str A &\in 2^{S^2},\\
\Phi &: 2^{S^2}\longrightarrow \mathbb{R}^3,\ \mbox{defined by}\\
\Phi(\str A) &= \left(\mbox{bounded,finite,length}\right)\in \mathbb{R}^3,\ \mbox{and}\\
f:2^{S^2} &\longrightarrow 2^{\mathbb{R}^4},\ \mbox{defined by}\\ 
f(2^{S^2}) &= \left\{\Phi(\str A)\in \mathbb{R}^3: \str A\in 2^{S^2}\right\}\in 2^{\mathbb{R}^4}.\mbox{\qquad \textcolor{blue}{$\blacksquare$}}
\end{align*}
\end{example}


\subsection{Brouwer's Fixed Point Theorem and Wired Friend Theorem}
Next, consider Brouwer's fixed point theorem~\cite{
Brouwer1906MathAnnFixedPoint,Brouwer1911MathAnnNo1FixedPoint,Brouwer1912MathAnnNo4FixedPoint,Brouwer1921MathAnnVol82FixedPoint}.
Let $B^n$ denote the unit $n$-ball in $\mathbb{R}^n$, which is the interior of a sphere $S^n$.  Let ${\bf x}$ be a point in $\mathbb{R}^n$ and let $\norm{{\bf x}} = \norm{(x_1,\dots,x_n)}$.  Then a \emph{closed unit ball} $B^n$\cite[\S 24]{Munkres2000} is defined by
\[
B_n = \left\{{\bf x}: \norm{{\bf x}} \leq 1\right\}.
\]
A \emph{fixed point} for a map $f:B_n\longrightarrow B_n$ is a point ${\bf x}$ so that $f({\bf x}) = {\bf x}$.

\begin{theorem}\label{thm:Brouwer}{Brouwer's Fixed Point Theorem for $n = 1,2$}{\rm \cite{Brouwer1906MathAnnFixedPoint}}
Let $n = 1,2$ and let $B_n$ be a closed unit ball in $\mathbb{R}^n$.   Then every continuous map $f:B_n\longrightarrow B_n$ 
has a fixed point.
\end{theorem}
\begin{proof}
The proof for $n = 1,2$ and closed unit ball $B_n$ in $\mathbb{R}^n$ is given in~\cite[\S 5, Theorem 5.1.22]{Runde2005topology}.
J. Hadamard gave a proof of this theorem for all dimensions~\cite{Hadamard1910BrowerFixedPointTheorem}.
\end{proof}

\noindent A complete study of Brouwer's Fixed Point Theorem is given by T. Stuckless~\cite{Stuckless1999MethodsOfProofForBrouwer}.

\begin{example}{\bf Coffee Cup Illustration of Brouwer's Fixed Point Theorem for Dimension 3}.\\
F.E. Su~\cite{Su1997AMMprovesBUT} gives a coffee cup illustration of the Brouwer Fixed Point Theorem (our Theorem~\ref{thm:Brouwer}) for all dimensions.  No matter how you continuously slosh the coffee around in a coffee cup, some point is always in the same position that it was before the sloshing began.   And if you move this point out of its original position, you will eventually move some \emph{other} point in the sloshing coffee back into into its original position. \qquad \textcolor{blue}{$\blacksquare$}
\end{example}

\begin{example}{\bf Proof of Brouwer's Fixed Point Theorem for Dimension $n$}.\\
F.E. Su~\cite{Su1997AMMprovesBUT} gives a constructive proof of the Brouwer Fixed Point Theorem (our Theorem~\ref{thm:Brouwer}) for all dimensions. V. Runde also gives a beautiful proof of Brouwer's Theorem for $n$ dimensions. \qquad \textcolor{blue}{$\blacksquare$}
\end{example}

We can always find a ball containing $g(\str A)\in \mathbb{R}^n$, which is a description of a string $\str A$ with $n$ features on a hypersphere $S^n$.  Also, observe that each $\str A$ has a particular shape (denoted by $\strShape A$).    

\begin{remark}{\bf String Theory and Wired Friends}.\\
The shape of a string $\strShape A$ is the silhouette of a string $\str A$.  
A $\str A$ is called a wired friend of an object $A$. Every wired friend is known by its shape.\qquad \textcolor{blue}{$\blacksquare$}
\end{remark}

These observations lead to a wired friend theorem.  The proof of Theorem~\ref{thm:wiredFriends} uses a projection mapping.
Let $O$ be an object space, $X$ a set of shapes in an $n$-dimensional space derived from $O$ and $Y$ a $k$-dimensional feature $\mathbb{R}^n$ space with $k \geq 1$.  Each component of a feature vector in $Y$ is a feature value of a shape in $X$.  Let $\longmapsto$ read \emph{maps to}.  Then a projection mapping is defined by
\[
O\ \longmapsto\ X\ \longmapsto\ Y.
\]

\begin{example}{\bf Sample Projection Mapping a Set of 2-dimensional (flat) Worldsheets}.\\
$\left\{\includegraphics[width=25mm]{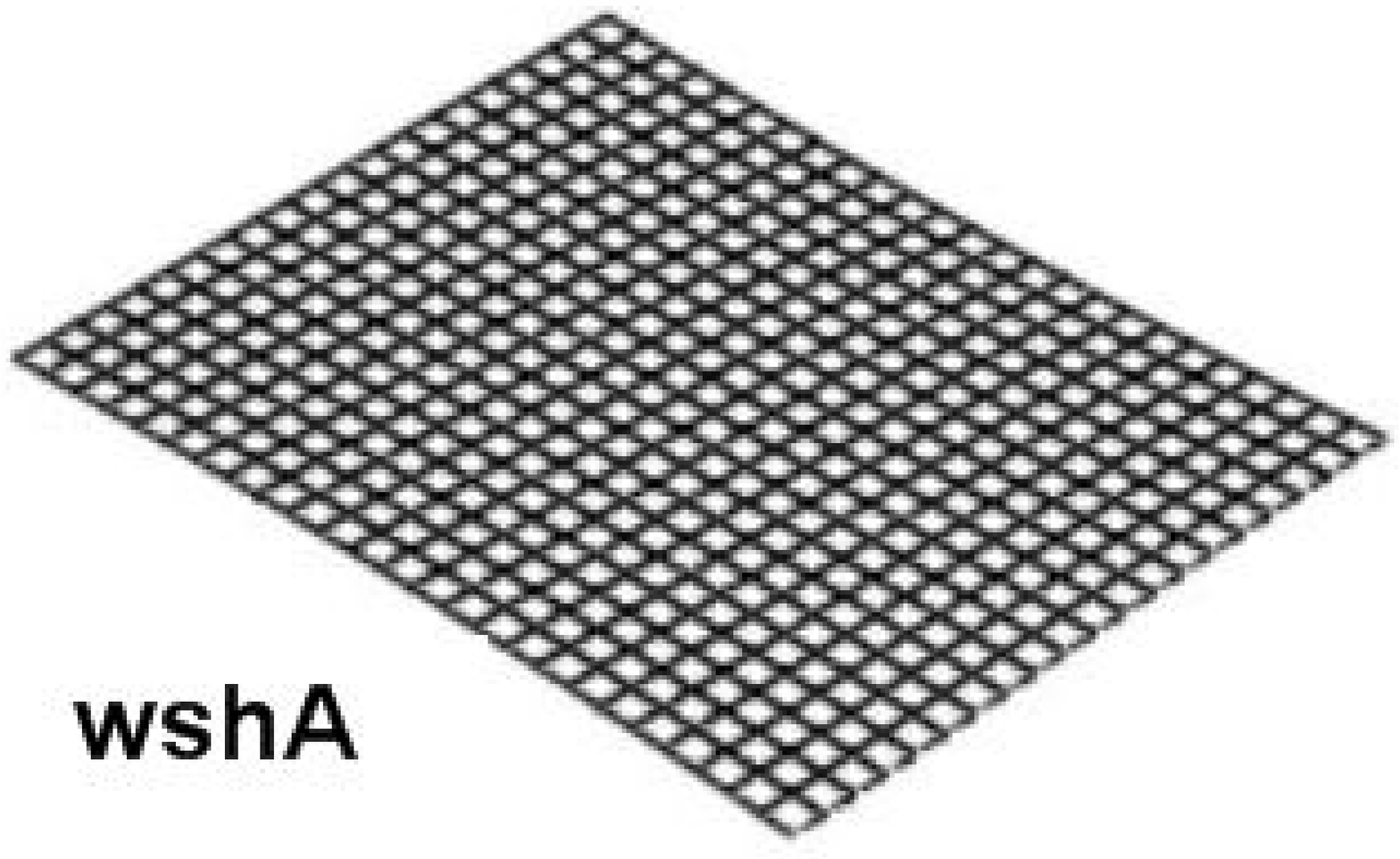}\right\}$\qquad{\large $\boldsymbol{\longmapsto}$}\qquad$\left\{\includegraphics[width=20mm]{2worldsheet}\right\}$\qquad{\large $\boldsymbol{\longmapsto}$}\qquad $\boldsymbol{\mathbb{R}^k}$.\\
Let $\wsh A\in O$ be a flat 2D worldsheet is a set of world sheets $W$.   Let $\cyl A\in X$ be a cylinder (rolled up worldsheet). And let $X$ be a set of $3$-dimensional worldsheet cylinders (flat worldsheets rolled up into a cylinder).  And let $Y$ be a $k$-dimensional feature space $\mathbb{R}^k$.  Each component of a feature vector in $\mathbb{R}^k$ is a feature value of a cylinder $\cyl A\in X$.  Then a projection mapping is defined by
\[
\left\{\wsh A\right\}\ \mathop{\longmapsto}\limits^f\ \left\{\cyl A\right\}\ \mathop{\longmapsto}\limits^g\ \mathbb{R}^k
\]  
Define $f(\wsh A) = \cyl A\in X$.  Select $k$ ($k \geq 1$) features of a cylinder $\cyl A$.  Then define 
\[
g(f(\wsh A)) = g(\cyl A) = k-\mbox{dimensional feature vector in}\ \mathbb{R}^k.
\]
\qquad \textcolor{blue}{$\blacksquare$}
\end{example}

\begin{theorem}\label{thm:wiredFriends}{\bf Wired Friend Theorem}\\
Every occurrence of a wired friend $\str A$ with a particular $\strShape$ with $k$ features on $S^n$ maps to a fixed description $g(\str A)$ that belongs to a ball $B_k$ in $\mathbb{R}^k$.
\end{theorem}
\begin{proof}
From Theorem~\ref{cor:strBUTsimplest}, we know $g(\str A) = g(\righthalfcap \str A)$ in $\mathbb{R}^k$ for every $\str A$ on $S^n$ or, for that matter (from Theorem~\ref{thm:re2reBUT}), in any
space $X$ containing $\str A$.  Every friend $A\in 2^X$ can be represented by a set of strings.  Observe that $g(\str A)$ is a fixed description of a $\strShape A$ of a wired friend $\str A$. In other
words, we have a projection mapping:
\[
\str A\ \mathop{\longmapsto}\limits^f\ \strShape A\ \mathop{\longmapsto}\limits^g\ \mbox{description}\ g(\strShape A), 
\]
{\em i.e.}, a wired friend $\str A$ maps to string shape $\strShape A$ such that $f(\str A) = \strShape A$.  And $\strShape A$ maps to a description of the string shape, namely $g(\strShape A)\in \mathbb{R}^k$. Since every $g(f(A)) = g(\strShape A)$ belongs to a $k$-dimensional ball $B_k$ (by definition), we have the desired result.
\end{proof}

\section{Application}
In this section, a sample application of strBUT is given in terms of Electroencephalography (EEG), which is an electrophysiological monitoring method used to record electrical activity in the brain.  The strBUT variant of the Borsuk-Ulam Theorem uses particles with closed trajectories instead of points.  The usual continuous function required by BUT is replaced by a proximally continuous function from region-based forms of BUT (reBUT)~\cite{PetersTozzi2016reBUT}~\cite[\S 5.7]{Peters2016ISRLcomputationalProximity}, providing the foundation for applications of strBUT.  This guarantees that whenever a pair of strings (spatial regions that are worldlines) are strongly close (near enough for the strings to have common elements), then we know that their mappings are also strongly close.  In effect, strongly close strings are strings with junctions that map to strongly close sets of vectors in $\mathbb{R}^n$.  

By way of application of the proposed framework, the closed paths described by strBUT represent biochemical pathways occurring in eukaryotic cells.   Indeed, in the cells of animals, the molecular components and signal pathways are densely connected with the rest of an animal's systems.   The tight coupling among different activities such as transcription, domain recombination and cell differentiation, gives rise to a signaling system that is in charge of receiving and interpreting environmental inputs.   Transmembrane molecular mechanisms continuously sense the external milieu, leading to amplification cascades and mobilization of many different actuators.  An intertwined, every changing interaction occurs between incoming signals and inner controlling mechanisms.  The string paths are closed and display a hole in their structure.  

In effect, real metabolic cellular patterns can be described in abstract terms as trajectories traveling on a donutlike structure.  Indeed, every biomolecular pathway in a cell is a closed string, intertwined with strong proximities~\cite[\S 1.4, 1.9]{Peters2016ISRLcomputationalProximity} with other strings to preserve the general homeostasis.

\setlength{\intextsep}{0pt}
\begin{wrapfigure}[11]{R}{0.35\textwidth}
\begin{minipage}{3.2 cm}
\centering
\includegraphics[width=37mm]{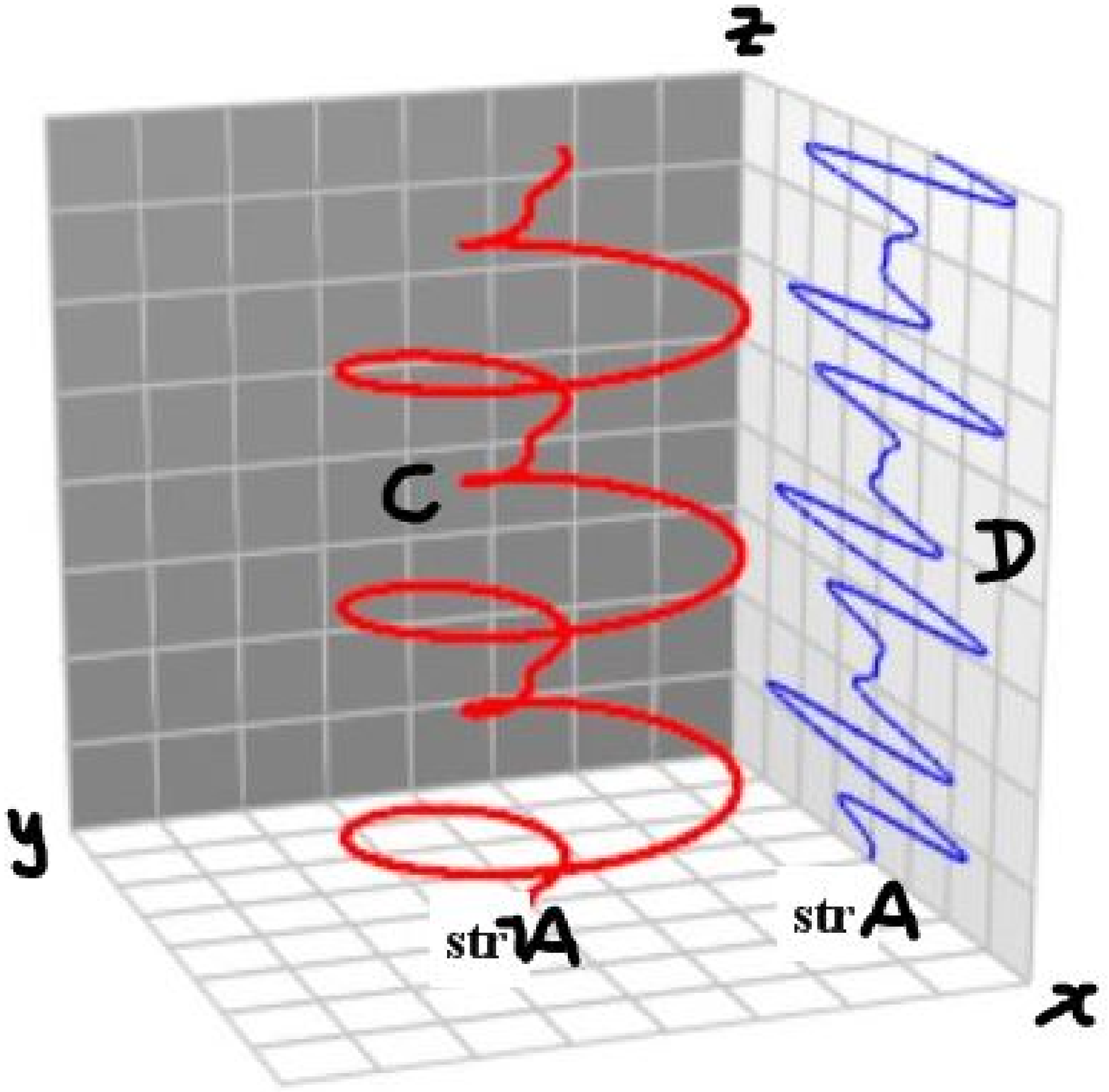}
\caption[]{\footnotesize $\boldsymbol{2^{S^2}\mapsto \mathbb{R}^3}$}
\label{fig:R2toR3}
\end{minipage}
\end{wrapfigure}
$\mbox{}$\\
\vspace{3mm}

This leads to an EEG scenario that fits into the framework described by Corollary~\ref{cor:stringSpace}. Let the path $D$ in Fig.~\ref{fig:R2toR3} represent a planar EEG trace and let $\str A\in 2^X$ be the string defined by the EEG trace, where $X = \mathbb{R}^2$ is a topological string space equipped with a strong proximity ${\sn}_X$.  In addition, let $R^3$ be a topological string space equipped with the the strong proximity ${\sn}_{\mathbb{R}^3}$. Further, $f:2^X\longrightarrow \mathbb{R}^3$ is a $\sn$-Re.s.p continuous mapping.  From Corollary~\ref{cor:stringSpace}, we can find the antipodal string $\righthalfcap \str A$ such that  $f(\str A) = f(\righthalfcap \str A)$.

\begin{example}\label{ex:twistingEEGsignal}
The proximally continuous mapping from a planar EEG string $\str A$ in $\mathbb{R}^2$ to an EEG string $f(\str A)$ in $\mathbb{R}^3$ is represented in Fig.~\ref{fig:R2toR3}.  In this example, $\str A$ is defined by an EEG signal following along the path $D$ in the $xz$-plane.  And $\str A \mapsto f(\str A)$ in 3-space, {\em i.e.}, there is a proximally continuous mapping from $\str A$ in 2-space to $f(\str A)$ in 3-space.  The added dimension is a result of considering the location of each point $(x,z)$ along the path $D$ as well as the twisting feature (call it $twist$) defined by 1.2*(1-cos(2.5*t))*cos(5*t) at time $t$.  That is, each point $(x,z)$ in $\str A$ corresponds to a feature vector $(x,z,twist)$ in $f(\str A)$ in 3-space.  A sample twist of the EEG signal is defined by twist(x,z) = 1.2(1-zcos(2.5x))cos(5x).
\qquad \textcolor{blue}{$\blacksquare$} 
\end{example}

\begin{figure}[!ht]
\centering
\includegraphics[width=55mm]{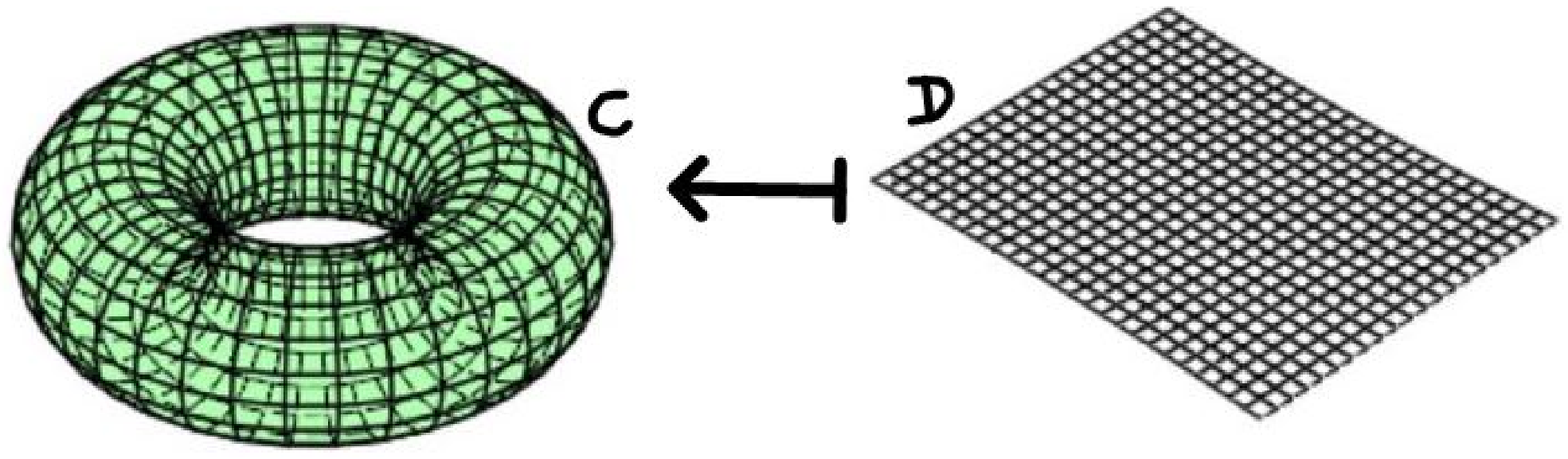}
\caption[]{$\boldsymbol{\mbox{\bf worldsheet}\ \sh D\mapsto\ \mbox{\bf torus}C}$}
\label{fig:R2toTorus}
\end{figure}
$\mbox{}$\\
\vspace{2mm} 

Combining the observations in Example~\ref{ex:rolledUpWorldsheet}(rolled up worldsheet to obtain a worldsheet cylinder) and Example~\ref{ex:bendedUpWorldsheetTorus} (bending a worldsheet cylinder to obtain a worldsheet torus), we obtain a new shape (namely, a torus) to represent the twists and turns of the movements of EEG signals along the surface of a torus (see Fig.~\ref{fig:R2toTorus}).

\begin{example}
\includegraphics[width=15mm]{0worldsheet}\qquad{\large $\boldsymbol{\mapsto}$}\qquad\includegraphics[width=10mm]{2worldsheet}
\quad{\large $\boldsymbol{\mapsto}$}\quad\includegraphics[width=10mm]{torus}\\
The proximally continuous mapping from a planar EEG worldsheet $\sh M$ in $\mathbb{R}^2$ to an EEG ring torus $f(\sh M)$ in $\mathbb{R}^3$ is represented in Fig.~\ref{fig:R2toTorus}.  A \emph{ring torus} is tubular surface in the shape of a donut, obtained by rotating a circle of radius $r$ (called the \emph{tube radius}) about an axis in the plane of the circle at distance $c$ from the torus center.  From example~\ref{ex:twistingEEGsignal}, each $\str A\in\sh M$ is defined by an EEG signal following along the path $D$ in the $xz$-plane.   And $\sh M$ maps to the tubular surface of a ring torus in 3-space, {\em i.e.}, there is a proximally continuous mapping from $\sh M$ in 2-space to ring torus surface $f(\sh M)$ in 3-space.   Then the surface area $S$ and volume $V$ of a ring torus~\cite[\S 8.7]{GellertEtAl1975torus} are given 
\begin{align*}
c &> r,\\
S &= 2\pi c\cdot 2\pi r = 4\pi^2 cr,\\
V &= 2\pi c\cdot \pi r^2 = 2\pi^2 cr^2.
\end{align*}
Carrying this a step further, the coordinates $x,y,z$ of the points in worldsheet strings on a ring torus can be expressed parametrically using the parametric equations in~\cite{Weisstein2016torus} as follows. 
\begin{align*}
u,v &\in [0,2\pi],c > r,\\
x &= (c + r\mbox{cos}\ v)\mbox{cos}\ u,\\
y &= (c + r\mbox{cos}\ v)\mbox{sin}\ u,\\
\mbox{twist}(r,v) &= r\mbox{sin}\ v. \mbox{\qquad \textcolor{blue}{$\blacksquare$} }
\end{align*}
\end{example}

\bibliographystyle{amsplain}
\bibliography{NSrefs}

\end{document}